\documentclass[11pt]{article}

\usepackage[margin=1in]{geometry}
\usepackage[numbers,sort&compress]{natbib}

\AddToHook{package/subcaption/after}{%

}

\usepackage{graphicx}
\usepackage{epstopdf}
\ifpdf
    \DeclareGraphicsExtensions{.eps,.pdf,.png,.jpg}
\else
    \DeclareGraphicsExtensions{.eps}
\fi

\usepackage{pgfplots}
\usepackage{pgfplotstable}
\pgfplotsset{compat=1.18}
\definecolor{oxfordnavy}{HTML}{1B2A4A}
\definecolor{signalorange}{HTML}{E8650A}
\definecolor{leafgreen}{HTML}{3DAA6E}
\usepgfplotslibrary{fillbetween}
\usepackage{amsmath,amssymb,amsfonts,amscd}
\usepackage{amsthm}
\usepackage[all,cmtip]{xy}
\usepackage{enumerate}
\usepackage{enumitem}
\usepackage{latexsym}
\usepackage{url}
\usepackage{nicematrix}
\usepackage{tikz}
\usepackage{algpseudocode,algorithm}
\usepackage{caption, subcaption}

\renewcommand{\p@subfigure}{}

\newenvironment{subfigure}[2][]{\begin{minipage}{#2}\refstepcounter{subfigure}\def\caption##1{\par\smallskip\centering\textbf{(\alph{subfigure})}~##1\par}}{\end{minipage}}
\usepackage{mathtools}
\usepackage{hyperref}
\usepackage{cleveref}
\usepackage{lineno}

\newcommand{\GN}{generalized Nyström}
\newcommand{\Ny}{Nyström}
\newcommand{\GNa}{generalized Nyström approximation}

\newcommand{\gnX}{\mathbf{A}_{\textup{GN}}}
\newcommand{\curX}{\mathbf{A}_{\textup{CUR}}}

\newcommand{\F}{\mathbb{F}}
\newcommand{\jcol}[2]{#1_{:, #2}} 
\newcommand{\jrow}[2]{#1_{#2, :}} 
\newcommand{\jk}[3]{[#1]_{#2, #3}} 
\newcommand{\njcol}[2]{#1_{:, -#2}} 

\definecolor{cSVD}{RGB}{0, 0, 0}
\definecolor{cGN}{RGB}{227, 178, 60}
\definecolor{cLTO}{RGB}{0, 20, 83}
\definecolor{cLPO}{RGB}{251, 97, 7}
\definecolor{cLCO}{RGB}{247, 37, 133}

\usepackage{amsopn}

\newtheorem{theorem}{Theorem}
\newtheorem{lemma}{Lemma}
\newtheorem{corollary}{Corollary}

\begin{document}
\title{Blind error estimation for CUR approximation}
\author{
    Lorenzo Lazzarino\textsuperscript{1}, Katherine J. Pearce\textsuperscript{2}, and Nathaniel Pritchard\textsuperscript{1}\\
    \small \textsuperscript{1}Mathematical Institute, University of Oxford, Oxford, OX2 6GG, UK\\
    \small \textsuperscript{2}Department of Mathematics and Statistics, University of New Mexico, Albuquerque, NM, USA\\
    \small Correspondence: \texttt{lorenzo.lazzarino@maths.ox.ac.uk}
}
\date{}
\maketitle
\frenchspacing

\begin{abstract}
    Low-rank approximation is a fundamental tool for scalable matrix computations. While such approximations have classically been formed via the truncated SVD, recent advances in randomized numerical linear algebra have produced methods of comparable accuracy at a fraction of the cost. A key advantage of these approaches is their ability to operate with limited access to the full matrix. A prominent example is the CUR decomposition, which builds an approximation from only a subset of columns and rows, making it particularly well-suited to settings where global matrix access is unavailable, such as in the design of spectro-microscopy experiments where even matrix-vector products cannot be computed.
    However, the randomness inherent to these methods introduces a key challenge: efficiently assessing approximation accuracy under the stringent constraint that only a subset of matrix entries can be observed. We derive a "blind" estimator of the Frobenius error of a CUR decomposition. For this estimator, we derive a worst-case lower bound on the number of required column queries, and provide techniques for quantifying its uncertainty. Finally, we demonstrate the performance of the estimator and the uncertainty sets on a mix of synthetic problems and real spectro-microscopy examples.
\end{abstract}


\noindent\textbf{Funding:} LL is supported by the Ada Lovelace Centre Programme at the Scientific Computing Department, STFC.

\section{Introduction}
Low-rank approximation is central to modern scientific computing, data science, engineering, and other computational sciences.
In these applications, matrices have become increasingly larger-scale and prohibitively expensive to use directly in computations. Worse, in a growing number of settings, matrices are not available as an explicit array or even as a routine for computing matrix-vector products: only individual entries of the matrix can be queried, often at significant cost (a leading example, spectro-microscopy, is discussed below).
Though the singular value decomposition (SVD) yields optimal low-rank decompositions for a given target rank \cite{eckart1936}, the cost of computing it scales poorly with the problem size, and it presupposes access to the full matrix, or at least the ability to apply it to vectors.
Alternative low-rank approximations are therefore preferable even when the matrix is fully accessible, purely on grounds of complexity and scaling; in the restricted access regime that motivates this paper, however, they become a necessity.

Among these alternatives, the CUR decomposition \cite{DrineasMahoney2009} is distinguished by its structure-preserving ``natural bases'' built from an approximate basis of the row and column spaces that consists of actual rows and columns of the input matrix, known as skeletons. Importantly, computing a CUR approximation does not necessarily require matrix-vector products.
This sets CUR apart from randomized SVD \cite{halko2011finding}, Nystr\"om \cite{williams20000}, and generalized Nystr\"om approximations \cite{nakatsukasaFastStableRandomized2020}, all of which intrinsically rely on applying the matrix to one or more random test matrices. In the restricted access settings we consider, CUR is consequently not merely convenient, but often the only available option.
Beyond this practical necessity, CUR decompositions preserve many properties of the input matrix, such as sparsity, non-negativity, interpretability in the context of the original data, and memory-efficiency, as only the skeleton indices, rather than dense factor matrices, need to be stored to reconstruct the row and column skeleton matrices.

 Moreover, it is possible to select skeletons so that the resulting approximation is quasi-optimal: the recent work of \cite{Osinsky2025} circumvents the issue of adversarial inputs and achieves an accuracy within a linear rank-dependent constant of the best error attained by the SVD.
This guarantee, however, comes at the price of computing an SVD of the input matrix, making the approach impossible in our setting, and impractical in most applications. In practice, greedy, adaptive, and often randomized skeleton-selection methods are relied upon instead, trading worst-case robustness for compatibility with restricted access \cite{CortinovisKressner2026, Deshpande2010, Drineas2006, Drmac2016, Frieze2004, Goreinov1997, DrineasMahoney2009, Osinsky2025, DongChenMartPearce2025, Pritchard2025, Sorensen2016, Stewart1999}.
A core remaining challenge is understanding how to certify the accuracy of a particular approximation, once such an approximation has been built, without any further access to the matrix beyond what restricted access already permits.

This question is not merely theoretical. A concrete and increasingly important example is spectro-microscopy, in which the spatial and spectral identities of materials are reconstructed from data collected during synchrotron experiments.
There, the acquisition of data, corresponding to individual entries of the associated matrix, is the true experimental bottleneck, and the reconstruction task can be reinterpreted as constructing an accurate CUR approximation of very low, but generally unknown, rank \cite{meier2026reducingacquisitiontimeradiation}.
Because every additional entry carries a significant acquisition cost (up to 5 days to form the full matrix in Spectro-microscopy \cite{hirose2019oxygendiffusiondriven}) reconstructing material identities in this way requires the support of error estimates of the CUR approximation without acquiring, or accessing, the entire matrix.
This application, together with related ones such as spectro-ptychography and spectro-tomography, motivates the \emph{blind} error estimators for CUR that we develop in this work: estimators that access the matrix only through a small subset of its entries, and never require access to the full matrix or to a matrix-vector product.

Randomization is valuable not only for constructing low-rank approximations, but also for the a posteriori estimation of their error once computed.
When matrix-vector products are available, sketching the error residual has been shown to work very well, even for a very small number of sketch vectors \cite[Sect. 12]{martinssonRandomizedNumericalLinear2020a}.
This tool, however, is unavailable in the restricted-access settings described above: with matrix-vector products out of reach and new individual entries costly to acquire, as in spectro-microscopy. Therefore, any viable estimator must, like the CUR approximation it certifies, make do with a small number of additional entries and nothing more.

A natural starting point is to ask whether the leave-one-out (LOO) framework for Nyström \cite{EpperlyTropp_VarianceEst2024} and generalized Nyström (GN) \cite{LazPearPrit26} approximations can be adapted to the closely related CUR approximation.
These LOO formulas yield to error estimates using only pre-computed quantities that are used to build the low-rank approximations.
Because the CUR decomposition may be viewed as a \GN{} decomposition in which random matrices are replaced by sub-matrices of the identity, one might hope that similar LOO formulas carry over to CUR.
However, as we will show in Section \ref{sec:CUR}, this seemingly small but important difference can cause the analogous LOO error estimates for CUR approximations to fail.

Inspired by the LOO approach and motivated by its limitations in the CUR setting, we propose a blind estimator of the CUR approximation error grounded in the theory of the Horvitz-Thompson estimation \cite{horvitz1952generalization}.
Namely, given a matrix $\mathbf{A}$ and its CUR approximation $\curX$ for some selected row and column skeletons, we show that randomly selecting a few additional columns of the error matrix $\mathbf{A}-\curX$ (crucially, columns not used in the construction of $\curX$) is often sufficient for estimating the error.
We prove that the square of our estimator is unbiased for the squared CUR approximation error, bound its variance, propose two ways to quantify the uncertainty of our estimator, provide a worst-case analysis, and demonstrate the robustness of its root-scale counterpart through numerical experiments on adversarial matrices and spectro-microscopy datasets.

\section{Preliminaries}

We begin by summarizing notation and key concepts from linear algebra that appear in our work.

\subsection{Notation}
\label{sec:notation}

We denote matrices by bold capital letters and vectors by bold lowercase letters, whereas scalars are lowercase.
For any matrix $\mathbf{M} \in \F^{m \times n}$, we  denote the $(i,j)$-th entry of $\mathbf{M}$ by $\mathbf{M}_{i,j}$.
We let $\jcol{\mathbf{M}}{j}$ denote the $j^\text{th}$ column of $\mathbf{M}$, and $\jrow{\mathbf{M}}{i}$ denote the $i^\text{th}$ row.
For notational convenience, we occasionally use the conventional Golub and Van Loan notation, e.g., $\mathbf{M}(:,j)$ vs. $\mathbf{M}_{:,j}$.
Given a matrix $\mathbf{M} \in \F^{m \times n}$,
we denote its
Frobenius norm by $\| \mathbf{M} \|_F$, so that
$$ \| \mathbf{M} \|_F = \sqrt{\sum_{i=1}^m \sum_{j=1}^n |\mathbf{M}_{i,j}|^2}.$$

We let $\mathbf{M}^*$ denote the conjugate transpose of $\mathbf{M}$ and $\mathbf{M}^\dagger$ its Moore-Penrose pseudoinverse.
We say a matrix $\mathbf{U}$ is \textit{orthonormal} if its columns are orthonormal so that $\mathbf{U}^* \mathbf{U} = \mathbf{I}$, the identity matrix.
The vector of all zeros is denoted by $\mathbf{0}$.

The abbreviation \textit{i.i.d.} stands for ``independent and identically distributed'' to describe random variables.
The probability of a random event is denoted by $\mathbb{P}[\cdot]$, and the expectation of a random variable is denoted by $\mathbb{E}[\cdot]$.
A random vector $\mathbf{x}$ is isotropic if $\mathbb{E}[\mathbf{x} \mathbf{x}^*
    ] = {\mathbf{I}}$.

\vspace{2mm}

\subsection{The Singular Value Decomposition}
\label{sec:svd}

Every matrix $\mathbf{A} \in \F^{m \times n}$ admits a {singular value decomposition} (SVD), given by $\mathbf{A} = \mathbf{U} \boldsymbol{\Sigma} \mathbf{V}^*$, matrices $\mathbf{U}$ and $\mathbf{V}$ are orthonormal, and $\boldsymbol{\Sigma}$ is diagonal.
The columns $\{ \mathbf{u}_i \}_{i=1}^{d}$ and $\{\mathbf{v}_{i}\}_{i=1}^{d}$ of $\mathbf{U}$ and $\mathbf{V}$ are called the left and right singular vectors of $\mathbf{A}$, where $d = \min(m,n)$.
The diagonal elements $\{ \sigma_i\}_{i=1}^{d}$ of $\boldsymbol{\Sigma}$ are the singular values of $\mathbf{A}$, ordered so that $\sigma_1 \geq \sigma_2 \geq \cdots \geq \sigma_d \geq 0$.
A rank-$k$ truncated SVD is given by $\mathbf{A}_k = \sum_{i=1}^k \sigma_i \mathbf{u}_i \mathbf{v}_i^*$.
By the Eckart-Young theorem \cite{eckart1936}, $\mathbf{A}_k$ gives the best rank-$k$ approximation of $\mathbf{A}$, with error
$$\| \mathbf{A} - \mathbf{A}_k\|_F^2 = \sum_{j = k+1}^{d} \sigma_j^2 .$$

\subsection{The Generalized Nystr\"om Decomposition}
\label{sec:N-GN-approx}
The \GN{} decomposition is an extension of the \Ny{} decomposition \cite{Gittens2016} to arbitrary matrices \cite{nakatsukasaFastStableRandomized2020}.
Given $\mathbf{A} \in \F^{m \times n}$, let $\mathbf{X} \in \F^{n \times r}$ and $\mathbf{Y} \in \F^{m \times s}$ be any two matrices (though frequently they are drawn from random matrix distributions).
The \GNa{} is given by
\begin{align}
    \label{eq:genNys}
    \mathbf{A} \langle \mathbf{X}, \mathbf{Y} \rangle \equiv \mathbf{A}_{\textup{GN}} = \left (\mathbf{A} \mathbf{X} \right ) \left ( \mathbf{Y}^* \mathbf{A} \mathbf{X} \right )^\dag \left ( \mathbf{Y}^* \mathbf{A} \right ).
\end{align}
The number of columns, $r$ and $s$, in the matrices $\mathbf{X}$ and $\mathbf{Y}$ need not be equal.
In fact, it has been suggested that introducing a discrepancy between $r$ and $s$ can enhance the accuracy of the \GNa~\cite{nakatsukasaFastStableRandomized2020}.
Throughout our work, we refer to the case where $r = s$ as the non-discrepant case, and the case where $s\geq r$ as the discrepant case.
We denote the core matrix $\mathbf{Y}^* \mathbf{A} \mathbf{X}$ by $\mathbf{H}$, which we assume is full-rank.

\subsection{The CUR Decomposition}
\label{sec:CUR}

The CUR decomposition is also frequently used for low-rank approximation of arbitrary $\mathbf{A} \in \F^{m \times n}$, denoted by
\begin{align}
    \label{eq:CUR}
    \curX = \underset{m \times r}{\mathbf{C}} \underset{r \times s}{\hspace{2mm} \mathbf{U} \hspace{2mm}} \underset{s \times n}{\mathbf{R}},
\end{align}
where $\mathbf{R} = \mathbf{A}(\mathcal{I},:)$ for $s$ linearly independent rows indexed by $\mathcal{I}$, and $\mathbf{C} = \mathbf{A}(:,\mathcal{J})$ for $r$ linearly independent columns indexed by $\mathcal{J}$.
The columns of $\mathbf{C}$ or rows of $\mathbf{R}$ are known as skeletons.
If $\mathbf{U} = \mathbf{C}^\dag \mathbf{A} \mathbf{R}^\dag$, then the decomposition in (\ref{eq:CUR}) is known as the CUR of ``best approximation'' (cf. \cite{ParkNakatsukasa25}), as this choice of $\mathbf{U}$ minimizes the Frobenius norm error for given $\mathbf{C}$ and $\mathbf{R}$.

However, in practice, it is prohibitively expensive, or impossible in some settings, especially when only entrywise access to the matrix is available, to access the full matrix $\mathbf{A}$ for  $\mathbf{U} = \mathbf{C}^\dag \mathbf{A} \mathbf{R}^\dag$.
Instead, we compute $\mathbf{U} = \mathbf{A}(\mathcal{I},\mathcal{J})^\dag$ to form the CUR decomposition in sublinear-time, with no additional matrix access beyond the skeletons; this is the decomposition computed in our work.

Unfortunately, $\mathbf{A}(\mathcal{I},\mathcal{J})^{\dag}$ can be highly ill-conditioned.
As such, the skeletons $\mathcal{I}$ and $\mathcal{J}$ should always be computed conditionally upon one another:
without loss of generality, given $\mathbf{C} = \mathbf{A}(:,\mathcal{J})$, the row skeletons should be computed using $\mathbf{C}$, not the full matrix $\mathbf{A}$.

\section{Leave-Right-Out Error Estimators}
With an understanding of the different low-rank approximation methods relevant to this work, we now review previous work on estimating the error  of Generalized Nystr\"om without additional matrix accesses, known as the leave-right-out estimator (LRO) \cite{LazPearPrit26}. We then follow this presentation with a description of how, despite the relationship between Generalized Nystr\"om and CUR, the LRO estimator does not extend to CUR.

\subsection{Leave-Right-Out for Generalized Nystr\"om}\label{subsec:LRO_GN}

Once a generalized \Ny{} approximation $\gnX$ of a given matrix $\mathbf{A}$ is obtained, computing the error $\|\mathbf{A} - \gnX \|_F$ is often impractical or infeasible.
For example, in matrix-free or streaming environments, approximating the error from the information used to compute the approximation is especially desirable.

We can approximate $\| \mathbf{A} - \gnX \|_F$ with the LRO error estimation technique of \cite{LazPearPrit26}, that is a generalization of the leave-one-out estimator presented in \cite{EpperlyTropp_VarianceEst2024} for the \Ny{} approximation.
In the LRO approach, for each of the $r$ columns of $\mathbf{X}$, we evaluate the error on the $i^\text{th}$ column of the rank-$(r-1)$ approximation obtained from excluding the $i^\text{th}$ column.
If implemented as described, this method would also be extremely costly computationally, since each error estimate with $r-1$ vectors requires an additional \GN{} approximation.
Instead, a fast LRO formula is derived for which no additional \GN{} approximations or matrix-vector products with $\mathbf{A}$ are needed beyond what is required to compute $\gnX$.

Suppose the core matrix $\mathbf{H}$ has full column rank.
Then we can estimate the error of the \GN{} approximation by leaving out a column $\mathbf{x}_\ell$ of $\mathbf{X}$ and summing over all indices; that is,
\begin{equation}
    \label{eq:NaiveLRO}
    \text{LRO} := \frac{1}{\sqrt{r}} \left( \sum_{\ell=1}^{r}  \left\Vert (\mathbf{A} - \njcol{(\mathbf{\gnX})}{\ell} ) \mathbf{x}_\ell \right\Vert^2 \right)^{\frac{1}{2}},
\end{equation}
where $\njcol{(\mathbf{\gnX})}{\ell}$ denotes the rank-$(r-1)$ generalized Nystr\"om approximation in which the $\ell$th column of $\mathbf{X}$ has been removed.
The LRO formula as in \eqref{eq:NaiveLRO} is  impractical, but in \cite{LazPearPrit26}, the following fast formula is derived:
\begin{equation}
    \label{eq:replicateLRO}
    \text{LRO} = \frac{1}{\sqrt{r}} \left( \sum_{\ell=1}^{r}  \left\Vert \mathbf{R_X} \frac{\jcol{(\mathbf{H}^*\mathbf{H})^{-1}}{\ell}}{\jk{(\mathbf{H}^*\mathbf{H})^{-1}}{\ell}{\ell}} \right\Vert^2 \right)^{\frac{1}{2}}
\end{equation}
where, $\mathbf{R_X}$ is the $R$-factor of the $QR$-factorization of $\mathbf{A}\mathbf{X}$, i.e., $[ \mathbf{Q}_\mathbf{X}, \; \mathbf{R}_\mathbf{X}] = \texttt{qr}(\mathbf{A}\mathbf{X})$. It is important to emphasize that, beyond the initial matrix vector products to form the approximation, at no point in this computation is the matrix $\mathbf{A}$ accessed. Yet, despite there being no additional accesses to the matrix $\mathbf{A}$, \cite{LazPearPrit26} showed experimentally that the LRO provides high-quality estimates of the true approximation error. These results align with the success of a similar approach proposed in \cite{EpperlyTropp_VarianceEst2024} for the \Ny{} approximation.

\subsection{Leave-Right-Out for CUR}
\label{subsec:LRO-cur}
\begin{figure}[h]
    \centering
    \begin{subfigure}{0.48\textwidth}
        \centering
        \scalebox{0.7}{
\begin{tikzpicture}
    \begin{axis}[
            width=\textwidth,
            height=8cm,
            xlabel={Rank},
            ymode=log,
            grid=both,
            legend to name={dummy},  
        ]

        \addplot[
            color=gray,
            solid,
            line width=1pt,
            mark=o,
            mark options={solid, scale = 0.7, fill = black},
        ]
        table[
                x=s_vec,
                y=errSVD,
                col sep=comma
            ]{csv-paper/lro-work/summary.csv};
        \addlegendentry{Spectrum}

        \addplot[
            color=oxfordnavy,
            solid,
            line width=1.5pt,
            mark=o,
            mark options={fill=white, scale = 2}
        ]
        table[
                x=s_vec,
                y=errCUR,
                col sep=comma
            ]{csv-paper/lro-work/summary.csv};
        \addlegendentry{CUR error}

        \addplot[
            color=leafgreen,
            dashed,
            line width=1.5pt,
            mark=x,
            mark options={solid, scale=2},
        ]
        table[
                x=s_vec,
                y=LRO,
                col sep=comma
            ]{csv-paper/lro-work/summary.csv};
        \addlegendentry{Estimate}

    \end{axis}
\end{tikzpicture}
}
        \caption{Exponential decay.}
        \label{fig:exp-decay-loo}
    \end{subfigure}
    \hfill
    \begin{subfigure}{0.48\textwidth}
        \centering
        \scalebox{0.7}{
\begin{tikzpicture}
    \begin{axis}[
            width=\textwidth,
            height=8cm,
            xlabel={Rank},
            ymode=log,
            grid=both,
            legend to name={dummy},  
        ]

        \addplot[
            color=gray,
            solid,
            line width=1pt,
            mark=o,
            mark options={solid, scale = 0.7, fill = black},
        ]
        table[
                x=s_vec,
                y=errSVD,
                col sep=comma
            ]{csv-paper/lro-fail/summary.csv};
        \addlegendentry{Spectrum}

        \addplot[
            color=oxfordnavy,
            solid,
            line width=1.5pt,
            mark=o,
            mark options={fill=white, scale = 2}
        ]
        table[
                x=s_vec,
                y=errCUR,
                col sep=comma
            ]{csv-paper/lro-fail/summary.csv};
        \addlegendentry{CUR error}

        \addplot[
            color=leafgreen,
            dashed,
            line width=1.5pt,
            mark=x,
            mark options={solid, scale=2},
        ]
        table[
                x=s_vec,
                y=LRO,
                col sep=comma
            ]{csv-paper/lro-fail/summary.csv};
        \addlegendentry{Estimate}

    \end{axis}
\end{tikzpicture}
}
        \caption{Frankenstein.}
        \label{fig:frankestein-loo}
    \end{subfigure}
    \caption{Estimator bias when the LRO formula of Section~\ref{subsec:LRO_GN} is directly extended to CUR approximations. \textit{Left:} In simple cases, such as when $A$ has exponentially decaying singular values, the LRO formula can provide a good estimate. \textit{Right:} For more difficult cases, such as the ``Frankenstein'' matrix, the bias of the formula prevents a good estimate. The ``Frankenstein'' input matrix is constructed as an $n \times n$ block diagonal matrix, where the $n/4 \times n/4$ blocks are the adversarial \texttt{`cauchy'}, \texttt{`golub'}, \texttt{`randcorr'}, and \texttt{`lotkin'} test matrices from JULIA's \texttt{MatrixDepot}. }
    \label{fig:CUR_LRO_frankenstein}
\end{figure}
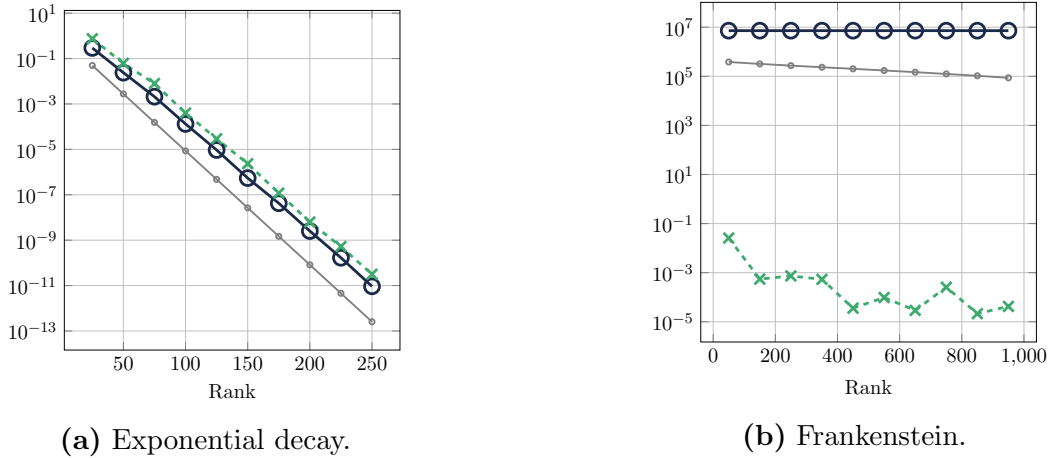

Because we can view the CUR decomposition as a special case of the \GN{} decomposition with canonical basis vectors that are a dependent subset of the canonical matrix, it is natural to ask whether an LRO estimator of the form (\ref{eq:replicateLRO}) may be derived for CUR as well.
Let $\mathcal{I}$ be the indices of the $s$ row skeletons and $\mathcal{J}$ be the indices of the $r$ column skeletons.
Now, in the CUR approximation, $\mathbf{X} \in \F^{n \times r}$ and $\mathbf{Y} \in \F^{m \times s}$ are  sub-matrices of the identity which select columns $\mathcal{J}$ and rows $\mathcal{I}$, respectively, so that $\mathbf{A} \mathbf{X} = \mathbf{C}:= \mathbf{A}(:,\mathcal{J})$ and $\mathbf{Y}^*\mathbf{A} = \mathbf{R} := \mathbf{A}(\mathcal{I},:)$.
Then, it would seem that we can directly apply the LRO estimator for \GN{} of the previous section to obtain an LRO estimator for CUR.

Unfortunately, the resulting estimator is not as reliable, as it is for \GN{}.
A probable cause of this loss in reliability is that for the LRO estimator, the columns of matrices $\mathbf{X}$ and $\mathbf{Y}$ are independent isotropic random vectors, cf. the Girard-Hutchinson trace estimator in \cite[Section 4.8]{martinssonRandomizedNumericalLinear2020a} or the leave-one-out \Ny{} error estimator in \cite{EpperlyTropp_VarianceEst2024}. Unfortunately, because CUR uses a subset of the columns and rows, the independence assumption for the random vectors no longer holds, as even when uniform sampling is performed, it requires that each column be selected without replacement. The effect of this dependency on the error calculation can be stronger or weaker based on the underlying matrix as is visualized in Figure~\ref{fig:CUR_LRO_frankenstein}.
For certain matrices, such as ones with exponentially decaying singular values, the estimator can still provide good estimates.
However, for more adversarial matrices, the estimator can completely fail.
For this reason, in the following section, we derive an alternative blind estimator.

\section{Blind error estimator}
Given the difficulties the LRO approach faces in the CUR setting, we propose here an alternative estimator built on a simple but effective observation: a few columns of the residual matrix $\mathbf{A} - \curX$ often suffice to accurately capture the magnitude of the full residual.
This suggests estimating the CUR approximation error by evaluating the residual on a small number of additional columns, not used in the construction of the approximation itself. Motivated by the Horvitz-Thompson estimators \cite{horvitz1952generalization}, an estimator in statistics commonly used in survey sampling to correct for missing data \cite{sarndal1992model}. Our estimator can accurately estimate the error with a small number of extra columns, and, as we will show, in its squared form is an unbiased estimator of the squared Frobenius error whose variance can be explicitly characterized in terms of both the extra column selection strategy and the quality of the CUR approximation.

Formally, given a rank-$r$ CUR approximation with selected row and column indices $\mathcal{I}$ and $\mathcal{J}$,
\begin{equation}
    \curX = \mathbf{A}(:,\mathcal{J}) \mathbf{A}(\mathcal{I},\mathcal{J})^\dag
    \mathbf{A}(\mathcal{I},:),
\end{equation}
we select $q$ extra columns indexed by $\mathcal{J}_q \subset \{1, \dots, n\} \setminus \mathcal{J}$. Then, defining $\mathbf{S}_q$ as the scaled submatrix of the identity matrix indexed by $\mathcal{J}_q$, we estimate the CUR error as
\begin{equation}
    \label{eq:estimator}
    \| \mathbf{A} - \curX\|_F \approx \|( \mathbf
    {A}-\curX) \mathbf{S}_q\|_F.
\end{equation}
The scalings in $\mathbf{S}_q$ will depend on the probability distribution from which we sample the indices in $\mathcal{J}_q$.
As illustrated in Figure \ref{fig:ill-estimate}, this is equivalent to computing the residual only on the selected (blue) extra columns.
Moreover, the interpolatory properties of the CUR approximation allow us to compute this estimate directly on the extra columns of the Schur complement, further reducing the number of components involved.
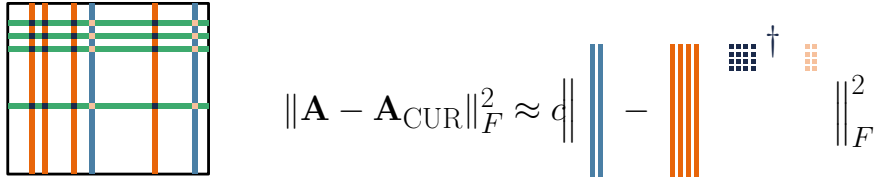
\begin{figure}[h]
    \centering
    \definecolor{leafgreen}{HTML}{3DAA6E}
\definecolor{oxfordnavy}{HTML}{1B2A4A}
\definecolor{signalorange}{HTML}{E8650A}
\definecolor{extracolblue}{HTML}{4A7FA5}   
\usetikzlibrary{calc}

\scalebox{0.7}{
\resizebox{\textwidth}{!}{%
  \begin{tikzpicture}
    \begin{scope}[xshift= 1.5cm, x=1cm, y=1cm, line cap=round, line join=round, scale=0.75, transform shape]
      \draw[very thick] (0.3,0.7) rectangle (4.3,4.1);
      \foreach \y in {3.7,3.45,3.2,2.05} {
          \fill[leafgreen] (0.3,\y-0.055) rectangle (4.3,\y+0.055);
        }
      \foreach \x in {0.78,1.05,1.62,3.23} {
          \fill[signalorange] (\x-0.055,0.7) rectangle (\x+0.055,4.1);
        }
      \foreach \x in {1.98,4.03} {
          \fill[extracolblue] (\x-0.055,0.7) rectangle (\x+0.055,4.1);
        }
      \foreach \x in {0.78,1.05,1.62,3.23} {
          \foreach \y in {3.70,3.45,3.20,2.05} {
              \fill[oxfordnavy]
              ($(\x,\y)+(-0.05,-0.05)$) rectangle ($(\x,\y)+(0.05,0.05)$);
            }
        }
      \foreach \x in {1.98,4.03} {
          \foreach \y in {3.70,3.45,3.20,2.05} {
              \fill[signalorange!40]
              ($(\x,\y)+(-0.05,-0.05)$) rectangle ($(\x,\y)+(0.05,0.05)$);
            }
        }
    \end{scope}
    \begin{scope}[xshift=5cm, x=1cm, y=1cm, line cap=round, line join=round, scale=0.6, transform shape]
      \node[anchor=west, font=\Huge] at (1.25,2.45)
      {$\|\mathbf{A-\curX}\|_F^2 \approx c$};
      \foreach \x in {9.1,9.3} {
          \fill[extracolblue] (\x-0.06,0.8) rectangle (\x+0.06,4.1);
        }
      \foreach \x in {11.1,11.3,11.5,11.7} {
          \fill[signalorange] (\x-0.06,0.8) rectangle (\x+0.06,4.1);
        }
      \node[font=\bfseries\Huge] at (10.2,2.45) {$-$};
      \node[font=\bfseries\Huge] at (8.5,2.45) {$\Big\Vert$};
      \foreach \i in {0,1,2,3} {
          \foreach \j in {0,1,2,3} {
              \fill[oxfordnavy]
              (12.5+0.18*\i,4.1-0.12-0.18*\j) rectangle +(0.12,0.12);
            }
        }
      \node[text=oxfordnavy,font=\bfseries\Huge] at (13.6,4.18) {$\dag$};
      \foreach \i in {0,1} {
          \foreach \j in {0,1,2,3} {
              \fill[signalorange!40]
              (14.4+0.18*\i,4.1-0.12-0.18*\j) rectangle +(0.12,0.12);
            }
        }
      \node[font=\bfseries\Huge] at (15.6,2.45) {$\Big\Vert_F^2$};
    \end{scope}
  \end{tikzpicture}%
}
}
    \caption{Illustration of the column-sampling estimator. \emph{Left:} the full
        matrix $\mathbf{A}$; green and red bars indicate the rows and columns selected by the CUR
        approximation, the violet squares their intersection; blue bars indicate the $q$ extra columns, and the yellow
        squares mark their intersections with the selected rows. \emph{Right:} we evaluate the residual only on the extra columns. Note: $c$ is a scaling factor, and for uniformly sampled extra columns, $c=\frac{n-r}{q}$.}
    \label{fig:ill-estimate}
\end{figure}
\paragraph{Choice of extra columns.}
A crucial component of the estimator is the choice of the extra columns. In particular, the quality of the estimator depends heavily on how this choice is made.
The available options depend largely on what is known about the problem: if no prior information is at hand, the only possibility is to sample the extra columns uniformly at random from those not already selected.
However, in most applications where matrix-vector products with $\mathbf{A}$ are unavailable and only part of the matrix is accessible (for example, when designing an experiment), some data-driven information is typically available \cite{meier2026reducingacquisitiontimeradiation}.
In such cases, one can improve the reliability of the estimator by exploiting this information to make a more informed selection of extra columns.

To keep the presentation general, we prove results for arbitrary probability weights on the extra columns, and present experiments with uniformly sampled extra columns to illustrate the behavior of the estimator in the limiting case where the only available information comes from the observed entries of $\mathbf{A}$.

\paragraph{Number of extra columns.}
Naturally, the quality of the estimate depends on the number $q$ of extra columns.
It is therefore important to understand how $q$ should scale with the number of rows of $\mathbf{A}$ and its rank.
To this end, we include two numerical experiments showing that the variance of the estimator \eqref{eq:estimator} relative to the true CUR approximation error is not strongly affected by increasing the number of rows at a fixed or increasing rank.
Indeed, Figure~\ref{fig:numer_q} shows that a small value of $q$, is sufficient for obtaining a reliable estimate, even for for matrices with large numbers of rows or large rank.

The reliability of the estimator is further examined in Section~\ref{sec:experiments}, where we assess its accuracy and discuss potential failure modes.
First, we present theoretical analysis along with an auxiliary strategy for assessing the quality of the estimate in practice.
\begin{figure}
    \centering
    \begin{subfigure}{0.48\textwidth}
        \centering
        \pgfplotsset{
    colormap={navygreen}{
            rgb255(0pt)=(27,42,74)       
            rgb255(33pt)=(15,110,86)     
            rgb255(66pt)=(61,170,110)    
            rgb255(100pt)=(232,101,10)   
            rgb255(133pt)=(253,246,227)  
        }
}

\scalebox{0.7}{
    \resizebox{\textwidth}{!}{%
        \begin{tikzpicture}
            \begin{axis}[
                    width=\textwidth,
                    height=\textwidth,
                    xlabel={rows of $A$},
                    ylabel={Extra columns},
                    xtick={1,2,...,10},
                    xticklabels={20,40,80,160,320,640,1280,2560,5120,10240},
                    xticklabel style={rotate=45, anchor=east, font=\small},
                    ytick={1,2,...,10},
                    yticklabels={1,2,4,8,16,32,64,128,256,512},
                    tick align=outside,
                    grid=both,
                    grid style={line width=0.3pt, draw=gray!30},
                    colorbar,
                    colorbar style={
                            width=0.2cm,
                        },
                    colormap name=navygreen,
                    point meta min=-9,
                    point meta max=-3,
                ]
                \addplot[
                    matrix plot*,
                    mesh/cols=10,
                    point meta=explicit,
                ] table [
                        x=x,
                        y=y,
                        meta=logval,
                        col sep=comma,
                    ] {csv-paper/extra-col-rel-var/summary.csv};
            \end{axis}
        \end{tikzpicture}
    }}
        \caption{Fixed rank.}
    \end{subfigure}
    \hfill
    \begin{subfigure}{0.48\textwidth}
        \centering
        \pgfplotsset{
    colormap={navygreen}{
            rgb255(0pt)=(27,42,74)       
            rgb255(33pt)=(15,110,86)     
            rgb255(66pt)=(61,170,110)    
            rgb255(100pt)=(232,101,10)   
            rgb255(133pt)=(253,246,227)  
        }
}

\scalebox{0.7}{
    \resizebox{\textwidth}{!}{%
        \begin{tikzpicture}
            \begin{axis}[
                    width=\textwidth,
                    height=\textwidth,
                    ylabel={Extra columns},
                    xlabel={rows of $A$},
                    xtick={1,2,...,10},
                    xticklabels={20,40,80,160,320,640,1280,2560,5120,10240},
                    xticklabel style={rotate=45, anchor=east, font=\small},
                    ytick={1,2,...,10},
                    yticklabels={1,2,4,8,16,32,64,128,256,512},
                    tick align=outside,
                    grid=both,
                    grid style={line width=0.3pt, draw=gray!30},
                    colorbar,
                    colorbar style={
                            width=0.2cm,
                        },
                    colormap name=navygreen,
                    point meta min=-9,
                    point meta max=-3,
                ]
                \addplot[
                    matrix plot*,
                    mesh/cols=10,
                    point meta=explicit,
                ] table [
                        x=x,
                        y=y,
                        meta=logval,
                        col sep=comma,
                    ] {csv-paper/extra-col-rel-var-increase-rank/summary.csv};
            \end{axis}
        \end{tikzpicture}
    }}
        \caption{Increased rank.}
    \end{subfigure}
    \caption{Relative variance of the estimate as the number of extra columns for the estimator as well as the total number of rows of a Chan matrix \cite{Chan1987} is increased.
        With the total number of columns of $\mathbf{A}$ fixed, each sub-figure illustrates the variance of the estimator for a matrix $\mathbf{A}$ of fixed rank (\textit{left}) and increasing rank (\textit{right}).}
    \label{fig:numer_q}
\end{figure}
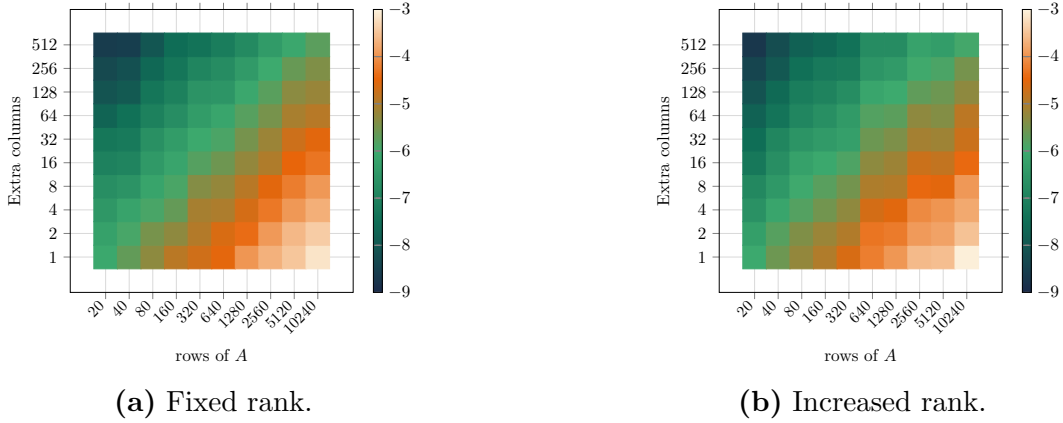
\subsection{Analysis}
\label{sec:analysis}
In this subsection, we provide some theoretical results on the quality of the estimator. To do this we begin by analyzing the squared form of our estimator, because that form aligns with the Horvitz-Thompson estimator \cite{horvitz1952generalization}. Then, we use the $\delta$-method \cite{oehlert1992note} to make statements about our estimator.
We start by proving that the square of our estimator is an unbiased estimate of the squared CUR approximation error.
We then proceed by analyzing the variance of the squared quantity, relating it to the probability distribution used for selecting the extra columns and to the quality of the CUR approximation. For simplicity all of our expectations are conditional on choosing an $\curX$, i.e. $\mathbb{E}[\cdot] = \mathbb{E}[\cdot | \curX]$

Our first theorem establishes that the proposed estimator squared is unbiased for the squared approximation error.
\begin{theorem} \label{thm:unbiased}
    Let $\curX$ be a rank-$r$ CUR approximation of an $m\times n$ matrix $\mathbf{A}$ with selected row and column indices $\mathcal{I}$ and $\mathcal{J}$, and let $r<n$.
    Let $\mathcal{J}_q\subset\{1,\dots,n\}\setminus\mathcal{J}$ be the set of $q$ extra columns, and define
    $\mathcal{S}:=\{\ell:\|(\mathbf{A}-\curX)\mathbf{e}_{\ell}\|_2>0\}$.
    Assume $p_{\ell}:=\mathbb{P}(\ell\in\mathcal{J}_q)>0$ for every $\ell\in\mathcal{S}$, and set $w_{\ell}=p_{\ell}^{-1/2}$. Define $\mathbf{S}_q=\mathbf{W}\mathbf{I}(:,\mathcal{J}_q)$ as the submatrix of the identity matrix indexed by $\mathcal{J}_q$, scaled by the diagonal matrix $\mathbf{W}$ with entries $w_{\ell}$.
    Then
    \begin{equation}
        \mathbb{E}[\|(\mathbf{A}-\curX)\mathbf{S}_q\|_F^2]=\|\mathbf{A}-\curX\|_F^2.
    \end{equation}
\end{theorem}
\begin{proof}
    Since residual columns outside $\mathcal{S}$ vanish,
    \begin{equation*}
        \mathbb{E}[\|(\mathbf{A}-\curX)\mathbf{S}_q\|_F^2]
        =\sum_{\ell\in\mathcal{S}}p_{\ell}w_{\ell}^2\|(\mathbf{A}-\curX)\mathbf{e}_{\ell}\|_2^2
        =\sum_{\ell\in\mathcal{S}}\|(\mathbf{A}-\curX)\mathbf{e}_{\ell}\|_2^2
        =\|\mathbf{A}-\curX\|_F^2.
    \end{equation*}
\end{proof}

We now study the variance of the estimates. To do so, we recall that all bounded random variables are sub-Gaussian (SG).
\begin{lemma}
    \label{lemma:boundedvar}
    \cite{Wainwright} If $Z$ is a random variable with $\mathbb{E}[Z] = \mu \neq 0$ and $\frac{Z - \mu}{\mu}$ takes values in $[z_1, z_2]$, then
    \begin{equation*}
        \frac{Z - \mu}{\mu} \sim \mathrm{SG}\!\left(\frac{z_2 - z_1}{2}\right).
    \end{equation*}
\end{lemma}
With a similar process to the one in \cite{pritchard2024solving}, we obtain the following result.
\begin{theorem}\label{thm:SG}
    Using the same notation as in Theorem~\ref{thm:unbiased}, let
    $\mathcal{S}:=\{\ell:\|(\mathbf{A}-\curX)\mathbf{e}_{\ell}\|_2>0\}$ denote the support of the nonzero residual columns.
    Assume that $p_{\ell}=\mathbb{P}(\ell\in\mathcal{J}_q)>0$ for every $\ell\in\mathcal{S}$, and set $w_{\ell}=p_{\ell}^{-1/2}$.
    Then
    \begin{equation}
        \frac{\| (\mathbf{A}-\curX)\mathbf{S}_q\|_F^2-\mathbb{E}[\|(\mathbf{A}-\curX)\mathbf{S}_q\|_F^2]}{\mathbb{E}[\|(\mathbf{A}-\curX)\mathbf{S}_q\|_F^2]}
        \sim \mathrm{SG}\!\left(\frac{\max_{\ell\in\mathcal{S}}w_{\ell}^2}{2}\right).
    \end{equation}
\end{theorem}
\begin{proof}
    We begin by observing that
    \begin{equation*}
        \frac{\|(\mathbf{A}-\curX) \mathbf{S}_q\|_F^2 - \mathbb{E}\left[\|(\mathbf{A}-\curX) \mathbf{S}_q\|_F^2\right]}{\mathbb{E}\left[\|(\mathbf{A}-\curX)\mathbf{S}_q\|_F^2\right]}
        = \frac{\|(\mathbf{A}-\curX)\mathbf{S}_q\|_F^2}{\mathbb{E}\left[\|(\mathbf{A}-\curX)\mathbf{S}_q\|_F^2\right]} - 1.
    \end{equation*}
    To apply Lemma~\ref{lemma:boundedvar}, it suffices to find a lower bound on $\mathbb{E}\left[\|(\mathbf{A}-\curX)\mathbf{S}_q\|_F^2\right]$.
    Indeed, since residual columns outside $\mathcal{S}$ vanish,
    \begin{equation*}
        \mathbb{E}[\|(\mathbf{A}-\curX)\mathbf{S}_q\|_F^2]
        =\sum_{\ell\in\mathcal{S}}p_{\ell}w_{\ell}^2\|(\mathbf{A}-\curX)\mathbf{e}_{\ell}\|_2^2
        =\|\mathbf{A}-\curX\|_F^2.
    \end{equation*}
    Then,
    \begin{equation*}
        \|(\mathbf{A}-\curX)\mathbf{S}_q\|_F^2
        =\sum_{\ell\in\mathcal{J}_q}w_{\ell}^2\|(\mathbf{A}-\curX)\mathbf{e}_{\ell}\|_2^2
        \leq \left(\max_{\ell\in\mathcal{S}}w_{\ell}^2\right)\|\mathbf{A}-\curX\|_F^2.
    \end{equation*}
    Therefore,
    \begin{equation*}
        -1\leq \frac{\|(\mathbf{A}-\curX)\mathbf{S}_q\|_F^2}{\mathbb{E}[\|(\mathbf{A}-\curX)\mathbf{S}_q\|_F^2]}-1
        \leq \max_{\ell\in\mathcal{S}}w_{\ell}^2-1.
    \end{equation*}
    Applying Lemma~\ref{lemma:boundedvar} proves the result.
\end{proof}

This bound provides a reasonable intuition for how the distribution of the estimator behaves based on the chosen sampling distribution on the column indices and their corresponding weights. Unfortunately, the estimator's bounding distribution can be highly conservative in practice, depending on the difference in magnitude of the column norms. Therefore, when it comes to estimating the uncertainty of the estimator we rely on a classical statistical approach known as the bootstrap (see \cref{sec:boot}).

For uniform sampling without replacement, a more direct characterization of the variance of the square of estimator is available. Let $N=n-r$ be the number of eligible columns. For the estimator based on $q$ sampled columns,
\begin{equation}\label{eq:uniform-var}
    \operatorname{Var}\!\left(\|(\mathbf{A}-\curX)\mathbf{S}_q\|_F^2\right)=\frac{N-q}{q(N-1)}\left(N\sum_{j=1}^{N}\|(\mathbf{A}-\curX)\mathbf{e}_j\|_2^4-\left(\sum_{j=1}^{N}\|(\mathbf{A}-\curX)\mathbf{e}_j\|_2^2\right)^2\right).
\end{equation}
This expression makes the dependence on $q$ and on the concentration of the residual energy across columns explicit.

\paragraph{Worst-case}
We now discuss how it is impossible for any estimator that looks only at the columns $\mathbf{A}$ to be accurate with high-probability for worst-case matrices.
Indeed, the constraint imposed in this work, namely estimating the CUR error without global access to $\mathbf{A}$, implies that there exist cases where reliable estimation is not possible without additional information.
A representative example is the following.
Consider a matrix $\mathbf{A}$ constructed as a $2 \times 2$ block diagonal matrix,
where the first block has very low rank and unit norm,
and the second diagonal block contains a scalar with large enough magnitude.
If the computed CUR has not selected the column corresponding to this scalar,
the large isolated block will very likely be missed by the estimator as well,
leading to a systematic failure of any randomized estimator that is allowed to only inspect columns of $\mathbf{A}$. To show this, we rely on Yao's minimax principle.
\begin{lemma}[Yao's minimax principle]
    Let $\mathcal{A}$ be a class of deterministic algorithms, $\mathcal{R}$ the associated class of randomized algorithms, and $c(\operatorname{alg},\mathbf{A})$ a cost function. Then
    \begin{equation}
        \min_{R\in\mathcal{R}}\max_{\mathbf{A}}
        \mathbb{E}\left[c(R,\mathbf{A})\right]
        =
        \max_{\mathcal{D}}\min_{\operatorname{alg}\in\mathcal{A}}
        \mathbb{E}_{\mathbf{A}\sim\mathcal{D}}
        \left[c(\operatorname{alg},\mathbf{A})\right].
    \end{equation}
\end{lemma}

This lemma states that the worst-case expected cost of the best randomized algorithm on a given problem equals the best expected cost of any deterministic algorithm against the hardest probability distribution over inputs. A discussion on how to use this lemma can be found in \cite{halikias2026transposefreelinearalgebra}. We use it to prove the following theorem, that applies to the case in which no prior information is available about the location of residual components not captured by the CUR approximation. In this setting, even allowing arbitrary randomized and adaptive selection of the extra columns cannot avoid a worst-case linear dependence on the number of possible locations of such a component.
\begin{theorem}
    Consider the family of matrices
    \begin{equation}
        \mathbf{A}^{(0)}
        =
        \begin{pmatrix}
            \mathbf{L} & \mathbf{0} \\
            0          & \mathbf{0}
        \end{pmatrix}
        \in\mathbb{R}^{m\times n},
        \qquad
        \mathbf{A}^{(j)}
        =
        \begin{pmatrix}
            \mathbf{L} & \mathbf{0}         \\
            0          & h\mathbf{e}_j^\top
        \end{pmatrix}
        \in\mathbb{R}^{m\times n},
        \qquad j=1,\ldots,d,
    \end{equation}
    where $\mathbf{L}\in\mathbb{R}^{\bar m\times\bar n}$ has rank one and
    $\|\mathbf{L}\|_F=1$, $m=\bar m+1$, $n=\bar n+d$, $h>0$, and
    $\mathbf{e}_j\in\mathbb{R}^d$ denotes the $j$-th canonical basis vector.
    Assume a CUR approximation is given, and that, for each $\mathbf{A}^{(j)}$,
    the column corresponding to $h$ is not captured by the approximation, i.e.,
    the CUR residuals are
    \begin{equation}
        \mathbf{A}^{(0)}-\mathbf{A}^{(0)}_{\mathrm{CUR}}
        =
        \begin{pmatrix}
            \mathbf{F}_L & \mathbf{0} \\
            0            & \mathbf{0}
        \end{pmatrix},
        \qquad
        \mathbf{A}^{(j)}-\mathbf{A}^{(j)}_{\mathrm{CUR}}
        =
        \begin{pmatrix}
            \mathbf{F}_L & \mathbf{0}         \\
            0            & h\mathbf{e}_j^\top
        \end{pmatrix},
        \qquad j=1,\ldots,d,
    \end{equation}
    where $\|\mathbf{F}_L\|_F=\varepsilon_L$. Define
    \begin{equation}
        \mathcal{T}
        :=
        \left\{
        \bar n+j:\ j\in\{1,\ldots,d\}
        \right\},
    \end{equation}
    and fix $\delta>1$. Assume that
    $h^2>(\delta^2-1)\varepsilon_L^2$. Finally, assume that all information available to the estimator prior to its additional column queries is identical across the matrices $\mathbf{A}^{(0)}, \dots,\mathbf{A}^{(d)}$

    Then, for every target success probability $\pi\in(1/2,1)$, there does not
    exist a randomized, possibly adaptive, algorithm that inspects at most
    \begin{equation}
        q < (2\pi-1)|\mathcal{T}|
    \end{equation}
    columns and outputs a $\delta$-approximation of
    $\|\mathbf{A}-\mathbf{A}_{\mathrm{CUR}}\|_F^2$ with probability at least
    $\pi$ for every input in
    $\{\mathbf{A}^{(0)},\mathbf{A}^{(1)},\ldots,\mathbf{A}^{(d)}\}$.
    In particular, any such algorithm requires
    $$
        q=\Omega(|\mathcal{T}|)=\Omega(d)
    $$
    column queries in the worst case.
\end{theorem}

\begin{proof}
    Consider the distribution $\mathcal{D}$ over the input matrices defined as
    follows. With probability $1/2$, draw $\mathbf{A}^{(0)}$, while with
    probability $1/2$, draw $\mathbf{A}^{(J)}$, where
    \begin{equation}
        J\sim\operatorname{Unif}\{1,\ldots,d\}.
    \end{equation}
    Equivalently,
    \begin{equation}
        \mathbb{P}_{\mathbf{A}\sim\mathcal{D}}
        \left(\mathbf{A}=\mathbf{A}^{(0)}\right)
        =\frac{1}{2},
        \qquad
        \mathbb{P}_{\mathbf{A}\sim\mathcal{D}}
        \left(\mathbf{A}=\mathbf{A}^{(j)}\right)
        =\frac{1}{2d},
        \qquad j=1,\ldots,d.
    \end{equation}

    Fix a deterministic algorithm that adaptively inspects at most $q$ columns.
    Run the algorithm on the input $\mathbf{A}^{(0)}$, and collect the indices
    of the inspected columns in $\mathcal{J}_q$. Since the algorithm is
    deterministic, $\mathcal{J}_q$ is fixed by the observations made on
    $\mathbf{A}^{(0)}$. Let
    \begin{equation}
        \mathcal{S}
        :=
        \mathcal{J}_q\cap\mathcal{T}
    \end{equation}
    denote the set of inspected columns among the $d$ possible locations of the
    additional component. Clearly,
    \begin{equation}
        |\mathcal{S}|\leq q.
    \end{equation}

    Now consider an input $\mathbf{A}^{(j)}$ with
    $\bar n+j\notin\mathcal{S}$. The only difference between
    $\mathbf{A}^{(0)}-\mathbf{A}^{(0)}_{\mathrm{CUR}}$ and
    $\mathbf{A}^{(j)}-\mathbf{A}^{(j)}_{\mathrm{CUR}}$ lies in column
    $\bar n+j$. Since this column is not inspected along the execution of the
    algorithm on $\mathbf{A}^{(0)}$, every column observed by the algorithm is
    identical under $\mathbf{A}^{(0)}$ and $\mathbf{A}^{(j)}$. Therefore, by
    induction over the adaptive queries, the deterministic algorithm makes the
    same sequence of queries and observes the same values on both inputs.
    Consequently, it must output the same value for $\mathbf{A}^{(0)}$ and
    $\mathbf{A}^{(j)}$.

    The target quantities are
    \begin{equation}
        \left\|
        \mathbf{A}^{(0)}-\mathbf{A}^{(0)}_{\mathrm{CUR}}
        \right\|_F^2
        =
        \varepsilon_L^2,
        \qquad
        \left\|
        \mathbf{A}^{(j)}-\mathbf{A}^{(j)}_{\mathrm{CUR}}
        \right\|_F^2
        =
        \varepsilon_L^2+h^2.
    \end{equation}
    Since
    $h^2>(\delta^2-1)\varepsilon_L^2$,
    we have
    \begin{equation}
        \delta\varepsilon_L^2
        <
        \frac{\varepsilon_L^2+h^2}{\delta}.
    \end{equation}
    Hence the two $\delta$-approximation intervals
    \begin{equation}
        \mathcal{I}_0
        =
        \left[
            \frac{\varepsilon_L^2}{\delta},
            \delta\varepsilon_L^2
            \right],
        \qquad
        \mathcal{I}_j
        =
        \left[
            \frac{\varepsilon_L^2+h^2}{\delta},
            \delta(\varepsilon_L^2+h^2)
            \right]
    \end{equation}
    are disjoint. Therefore, whenever $\bar n+j\notin\mathcal{S}$, the common
    output of the deterministic algorithm on $\mathbf{A}^{(0)}$ and
    $\mathbf{A}^{(j)}$ can be a $\delta$-approximation for at most one of these
    two inputs.

    We now bound the average success probability of the deterministic algorithm
    under $\mathcal{D}$. Let $a_0\in\{0,1\}$ indicate whether the algorithm is
    successful on $\mathbf{A}^{(0)}$. For each $j$ such that
    $\bar n+j\notin\mathcal{S}$, the disjointness of the two approximation
    intervals implies that, if $a_0=1$, the algorithm must fail on
    $\mathbf{A}^{(j)}$. For indices satisfying $\bar n+j\in\mathcal{S}$, we
    upper bound the success probability by one. Therefore, if the algorithm is
    successful on $\mathbf{A}^{(0)}$, its success probability under
    $\mathcal{D}$ is at most
    \begin{equation}
        \frac{1}{2}
        +
        \frac{1}{2}\frac{|\mathcal{S}|}{d}
        \leq
        \frac{1}{2}
        +
        \frac{q}{2d}.
    \end{equation}
    If the algorithm is unsuccessful on $\mathbf{A}^{(0)}$, then even assuming
    that it succeeds on every $\mathbf{A}^{(j)}$, its success probability under
    $\mathcal{D}$ is at most $1/2$. Thus, in either case,
    \begin{equation}
        \mathbb{P}_{\mathbf{A}\sim\mathcal{D}}
        \left(
        \operatorname{alg}\text{ successfully produces a }
        \delta\text{-approximation}
        \right)
        \leq
        \frac{1}{2}
        +
        \frac{q}{2d}.
    \end{equation}

    It follows that if
    \begin{equation}
        q < (2\pi-1)d
        =
        (2\pi-1)|\mathcal{T}|,
    \end{equation}
    then every deterministic algorithm using at most $q$ column queries has
    average success probability strictly smaller than $\pi$ under
    $\mathcal{D}$. By Yao's minimax principle, the same lower bound applies to
    randomized algorithms: there does not exist a randomized, possibly adaptive,
    algorithm using at most $q$ column queries that achieves success probability
    at least $\pi$ on every input in the support of $\mathcal{D}$. Since
    $$
        \operatorname{supp}(\mathcal{D})
        \subseteq
        \{
        \mathbf{A}^{(0)},\mathbf{A}^{(1)},\ldots,\mathbf{A}^{(d)}
        \},
    $$
    this proves the claim.
\end{proof}

Situations of this kind are unlikely to arise in the practical settings that are of primary interest in this work. More importantly, the preceding worst-case result concerns the setting in which no prior information is available about the location of residual components not captured by the CUR approximation. When such information is available, the corresponding information model changes, and the lower bound above does not directly apply. In particular, prior or data-driven information can be incorporated into our estimator through the probability distribution used to select the extra columns, assigning larger probability to columns that are more likely to contain significant residual components.

In a real-world application, the adversarial construction above would typically reflect either a flawed experimental design or a situation in which the user possesses such additional information about the structure of $\mathbf{A}$. For example, in spectro-microscopy, the first case would represent stacking two different samples one after the other, meaning trying to perform two different experiments at once. In the latter case, a data-driven approach to the selection of extra columns would naturally handle what we have called the worst case, rendering it no more difficult than the other settings considered in this section.
\subsection{The bias of our estimator}
It is important to note that while all of these results focus on the $\|\cdot\|_F^2$, our quantity of interest is actually the $\|\cdot\|_F$ estimator. From Jensen's inequality, we know that $  \mathbb{E}[\sqrt x] \leq \sqrt{\mathbb{E}[x]}$, meaning that the norm of the subset of the error will provide a slight underestimate of the norm of the true error. We can approximate this bias by using a second order Taylor expansion of the square root function. For a general random variable $x$ with expectation $\mu$ and variance $\xi^2$ we know that
\begin{align}
    \mathbb{E}[\sqrt{x}] & \approx \mathbb{E}\left[\sqrt{\mu} + \frac{(x - \mu)}{2 \sqrt{\mu}} - \frac{(x - \mu)^2}{8 \mu^{3/2}}\right] \\
                         & = \sqrt{\mu} - \frac{\xi^2}{8 \mu^{3/2}}
\end{align}
Using this and that $\mu = \|\mathbf{A} - \curX\|_F^2 = \mathbb{E}[\|(\mathbf{A} - \curX)\mathbf{S}_q\|_F^2]$ from \cref{thm:unbiased}, we can conclude that
\begin{equation}
    \|\mathbf{A} - \curX\|_F - \mathbb{E}[\|(\mathbf{A} - \curX)\mathbf{S}_q\|_F] \approx \frac{\xi^2 }{8\|\mathbf{A} - \curX\|_F^3}.
\end{equation}
Here we obtain that the bias of the estimator for the norm is dependent on the true error of the solution and the variance of the squared-error estimator. The latter will depend on the number of extra and the residual magnitude across columns.

With the approximate finite characterization of the bias, it is also useful to provide a high probability characterization of that bias. We obtain this high probability statement as a corollary to \cref{thm:SG}.

\begin{corollary}
    Under the conditions of \cref{thm:SG},  we have
    \begin{equation}
        \mathbb{P}\left( \frac{\|(\mathbf{A} - \curX)\mathbf{S}_q\|_F -\|\mathbf{A} - \curX\|_F}{\|\mathbf{A} - \curX\|_F} \geq \epsilon \right) \leq 2 \exp \left(-\frac{2(\epsilon^2 - 2\epsilon)^2 }{(\max_{\ell\in\mathcal{S}}w_{\ell}^2)^2} \right)
    \end{equation}
\end{corollary}
\begin{proof}
    For any $1>\epsilon >0$, the probability that
    \begin{align}
         & \mathbb{P}\left( \|(\mathbf{A} - \curX)\mathbf{S}_q\|_F \leq (1-\epsilon) \|\mathbf{A} - \curX\|_F \right)                                                         \\
         & =  \mathbb{P}\left( \|(\mathbf{A} - \curX)\mathbf{S}_q\|_F^2 \leq (1-\epsilon)^2 \|\mathbf{A} - \curX\|_F^2 \right)                                                \\
         & =  \mathbb{P}\left( \frac{\|(\mathbf{A} - \curX)\mathbf{S}_q\|_F^2 - \|\mathbf{A} - \curX\|_F^2}{\|\mathbf{A} - \curX\|_F^2 } \leq (\epsilon^2 - 2\epsilon)\right) \\
         & =  \mathbb{P}\left( \frac{\|\mathbf{A} - \curX\|_F^2 - \|(\mathbf{A} - \curX)\mathbf{S}_q\|_F^2}{\|\mathbf{A} - \curX\|_F^2} \geq (2\epsilon - \epsilon^2) \right) \\
         & \leq \exp \left(-\frac{2(\epsilon^2 - 2\epsilon)^2 }{(\max_{\ell\in\mathcal{S}}w_{\ell}^2)^2} \right),
    \end{align}
    where the last line uses the sub-Gaussian distribution from \cref{thm:SG}. Similarly, if $\epsilon > 0$ we have
    \begin{align}
         & \mathbb{P}\left( \|(\mathbf{A} - \curX)\mathbf{S}_q\|_F \geq (1+\epsilon) \|\mathbf{A} - \curX\|_F \right)                                                    \\
         & = \mathbb{P}\left( \|(\mathbf{A} - \curX)\mathbf{S}_q\|_F^2 \geq (1+\epsilon)^2 \|\mathbf{A} - \curX\|_F^2 \right)                                            \\
         & =\mathbb{P}\left( \|\frac{(\mathbf{A} - \curX)\mathbf{S}_q\|_F^2 - \|\mathbf{A} - \curX\|_F^2}{\|\mathbf{A} - \curX\|_F^2} \geq 2\epsilon +\epsilon^2 \right) \\
         & \leq \exp \left(-\frac{2(\epsilon^2 + 2\epsilon)^2 }{(\max_{\ell\in\mathcal{S}}w_{\ell}^2)^2} \right)
    \end{align}
    Combing these two bounds will give us
    \begin{equation}
        \mathbb{P}\left( \frac{|\|(\mathbf{A} - \curX)\mathbf{S}_q\|_F - \|\mathbf{A} - \curX\|_F  |}{ \|\mathbf{A} - \curX\|_F }\geq \epsilon \right) \leq 2 \exp \left(-\frac{2(\epsilon^2 - 2\epsilon)^2}{(\max_{\ell\in\mathcal{S}}w_{\ell}^2)^2} \right)
    \end{equation}
\end{proof}

\section{Quantifying the distribution of the estimator} \label{sec:uq}
Because the value of the estimator depends on the sample, a different sample of columns could mean a drastically different value of the estimate.
This variation in results can lead to incorrect determinations about whether a particular CUR is of high enough quality.
One way we can improve the reliability of our decisions based on an estimator is to obtain uncertainty intervals, which are likely to contain the value of the true error.
In order to compute these intervals, we need to know or estimate the distribution of the error across all possible subsets of a given size.
In statistics, estimating these distributions is known as uncertainty quantification \cite{pritchard2023practical}.

One possible approach is to estimate the maximal CUR error across all possible choices of unselected columns, in a manner similar to \cite{pritchard2023practical}.
Once this estimate is obtained, we can then plug in that quantity as the parameter of a sub-Gaussian distribution and derive the confidence intervals using that estimated distribution.
There are two problems with this approach.
First, if we lack global access to the full matrix, it is very difficult to determine the maximal error.
Second, even if we know the maximal error over all possible subsets, the sub-Gaussian distribution provides a highly conservative estimate \cite{pritchard2024solving}.
\subsection{Gaussian Estimation}
Another approach to uncertainty quantification of the Frobenius norm of the error can be found in the well established theory for the Horvitz-Thompson estimator. Specifically, it can be found in the theory that establishes the fact that the Horvitz-Thompson estimator is asymptotically normal \cite{berger1998rate}, in the sense that the number of columns and the number of extra columns grow to infinity. This means that in many scenarios, provided the number of total columns and extra columns is large enough, we can approximate the squared norm estimator with the normal distribution,
\begin{equation}
    \|(\mathbf{A} - \curX)\mathbf{S}_q\|^2_F - \|\mathbf{A} - \curX\|^2_F  \sim \mathcal{N}(0, \operatorname{Var}\!\left(\|(\mathbf{A}-\curX)\mathbf{S}_q\|_F^2\right)).
\end{equation}

Additionally, by using the delta method \cite{oehlert1992note}, we can find that the estimator for the Frobenius norm will be distributed
\begin{equation}\label{eq:asym_dixt}
    \|(\mathbf{A} - \curX)\mathbf{S}_q\|_F - \|\mathbf{A} - \curX\|_F  \sim \mathcal{N}\left(0, \frac{\operatorname{Var}\!\left(\|(\mathbf{A}-\curX)\mathbf{S}_q\|_F^2\right)}{4\|\mathbf{A} - \curX\|_F^2}\right).
\end{equation}
Using this result in the uniform sampling case, we can estimate \cref{eq:uniform-var} using the following sample variance formula for a Horvitz-Thompson estimator,
\begin{equation}\label{eq:varaince-approx}
    \widehat{\text{Var}} = \frac{q(N-q)}{4N(q-1)\|(\mathbf{A} - \curX)\mathbf{S}_q\|_F^2} \sum_{i=1}^q \left(\|(\mathbf{A} - \curX)\mathbf{S}_q(:,i)\|_2^2 - \frac{\|(\mathbf{A} - \curX)\mathbf{S}_q\|_F^2}{q} \right)^2.
\end{equation}

With this variance estimate, we can then form an uncertainty interval using $z-$intervals from classical statistics. Specifically, we can form the $(1 - \alpha)$ interval by finding the $1 - \alpha/2$ and $\alpha/2$ percentiles of \cref{eq:asym_dixt} with the variance approximation in \cref{eq:varaince-approx} used in place of the true variance. Then we can add and subtract those values from the estimator to produce the final interval. For instance, if we wanted a 95\% interval we would obtain,
\begin{equation}
    \left[\|(\mathbf{A} - \curX)\mathbf{S}_q\|_F-1.96\sqrt{\widehat{\text{Var}}},\text{ } \|(\mathbf{A} - \curX)\mathbf{S}_q\|_F +1.96\sqrt{\widehat{\text{Var}}}\right].
\end{equation}

It is important to note that the Gaussian distribution only holds asymptotically; thus, the above approach is an approximation that is only accurate when the distribution of the column norms of the residual is close to the Gaussian, which depending on the matrix may or may not hold true.

\subsection{Bootstrap Estimation}\label{sec:boot}

Another approach to quantifying the uncertainty involves using a common statistical technique known as the bootstrap \cite{efron1979bootstrap}.
There is prior work in randomized numerical linear algebra that uses the bootstrap or related techniques, e.g., \cite{lopes2020error, Lopes2018Error, EpperlyTropp_VarianceEst2024, LazPearPrit26}, but our use of the bootstrap is more classical.
At its core, the bootstrap is used to estimate the distribution of a particular function, $t(\cdot)$, over all possible samples of size $q$ from an unknown population $\mathcal{P}$.
In the context of our estimator, the relevant function is  $t(\{z_1, \dots, z_q\}) = {\sum_{i=1}^q z_i}$, where $z_i = \|(\mathbf{A}-\curX) \mathbf{S}_{q}(:, i)\|_2^2$ are the sample entries.
The bootstrap then estimates the distribution of $t(\{z_1, \dots, z_q\})$ across all possible samples of size $q$  in the population, by treating $\{z_1, \dots, z_q\}$ as the population and then repeatedly sampling with replacement from this population.
Once these bootstrap samples are obtained, if we let $z_i^{(j)}$ correspond to $i^\text{th}$ observation of the $j^\text{th}$ bootstrap sample, we can evaluate $t(\{z^{(j)}_1, \dots,  z^{(j)}_q\}) - t(\{z_1, \dots, z_q\})$ for each bootstrap sample. The distribution of these differences will then approximate the distribution of the $t(\{z_1, \dots, z_q\})$ across all possible choices of $z_1, \dots, z_q$ \cite{efron1979bootstrap}. We can therefore use quantiles of this distribution to compute confidence intervals for $t(\{z_1, \dots, z_q\})$. 

To measure the uncertainty of our error estimator, we compute the bootstrap according to the procedure in \Cref{alg:boot}.
In this algorithm, we take as inputs the number of bootstrap iterations $b$, the confidence level $1 - \alpha$ to cover with our interval, and our original column residual norms squared given by $z_i = \|(\mathbf{A}-\curX)\mathbf{S}_{q}(:, i)\|_2^2$.
We first define the relevant function (line 1) and perform a standard bootstrap procedure to collect bootstrap samples $\hat \theta_{(i)}$ (lines 3-4).
Finally, we obtain the desired interval by finding the $\alpha/2$ and $1-\alpha/2$ percentiles of those differences, then subtracting those quantities from the initial error estimate.
This bootstrap interval can experience error from two sources. First, it has error from the Monte Carlo sampling used to approximate the distribution. This error will decrease as $b \to \infty$. Second, there is error associated with the validity of the bootstrap approximation itself, which can be reduced by increasing the number of extra columns selected. Finally, to match the scale of the $\|\cdot\|_F$ we return the square root of the lower and upper bound as our interval.

\begin{algorithm}
    \caption{Bootstrap Error Confidence Interval} \label{alg:boot}
    \begin{algorithmic}[1]
        \Require $b, \{z_1, \dots, z_q\}, \alpha$
        \State $ \bar\theta = { \sum_{j=1}^q z_j}$
        \For{$i = 1:b$}
        \State Sample with replacement,
        $\bar Z_i = \{\bar z_1^{(i)}, \dots \bar z_q^{(i)}\}$
        \State Compute $\hat \theta_{(i)} = \sum_{j=1}^q \bar z_j^{(i)}$
        \EndFor
        \State $\Theta= \{(\hat \theta_{(1)} - \bar \theta), \dots, (\hat \theta_{(b)} - \bar \theta)\}$\\
        \Return $[\sqrt{\bar \theta - \Theta_{1-\alpha/2}}, \sqrt{\bar \theta - \Theta_{\alpha/2}}]$
    \end{algorithmic}
\end{algorithm}


While the bootstrap typically works, it does require the observations in the sample to be independent from one another and the distribution to not have extreme tails \cite{wasserman2008all}.
While sampling extra columns without replacement introduces a small amount of dependency, seemingly contradicting the first assumption, when the number of extra columns is small relative to the total number of columns, the impact of this dependency will be minimal. One could adjust for this dependency using a finite sample bootstrap approaches\cite{mccarthy1985bootstrap}. The finite sample approach computes the variance of the $\hat{\theta}_{(i)}$s in \cref{alg:boot} and scales them by $(n-r - q) / n-r$ \cite{rao1988resampling}. If we consider a rank 100 approximation to a matrix of size 10,000 with 10 extra columns this adjustment would be $9890/9900 = .9989$, in which case the finite sample bootstrap is essentially the same as the standard bootstrap.

The other assumption that can be problematic is the long tail probabilities, as a matrix with one error of large norm and the rest of small norm will not be accounted for in the bootstrap if it is not chosen in the initial set of extra columns.
This shortcoming is a consequence of not being able to see the whole matrix and cannot be overcome without global access, as previously discussed in Section~\ref{sec:analysis}.

\section{Numerical Experiments}
\label{sec:experiments}
In this section, we illustrate the performance of the proposed estimator through a series of numerical experiments.
Specifically, we assess the accuracy of the estimator on simple, challenging, and practical examples,
demonstrate how bootstrap can be used to quantify the reliability of the estimate in practice,
and conclude with a discussion of the worst-case scenario.

\subsection{Accuracy of the estimator}

As a first set of experiments, we assess the practical accuracy of the proposed square-root estimator.
To this end, we consider three experiments of increasing difficulty, illustrating the behaviour of the estimator across different regimes.
In the following experiments, errors are reported on the original Frobenius-norm scale. Specifically, we plot the true Frobenius error and the square root of the unbiased estimator of the squared Frobenius error. Since the square root is applied before averaging, the reported averages are of these root-scale quantities. These experiments therefore assess the practical accuracy of the square-root estimator.

\paragraph{Exponentially decaying singular values.}
\begin{figure}
    \centering
    \begin{subfigure}{0.48\textwidth}
        \centering
        \scalebox{0.7}{
\begin{tikzpicture}
    \begin{axis}[
            width=\textwidth,
            height=8cm,
            xlabel={Rank},
            ymode=log,
            grid=both,
            legend pos = south west
        ]

        \addplot[
            color=gray,
            solid,
            line width=1pt,
            mark=o,
            mark options={solid, scale = 0.7, fill = black},
        ]
        table[
                x=s_vec,
                y=errSVD,
                col sep=comma
            ]{csv-paper/quality-exp6/summary.csv};
        \addlegendentry{Spectrum}

        \addplot[
            color=oxfordnavy,
            solid,
            line width=1.5pt,
            mark=o,
            mark options={fill=white, scale = 1.5}
        ]
        table[
                x=s_vec,
                y=errCUR,
                col sep=comma
            ]{csv-paper/quality-exp6/summary.csv};
        \addlegendentry{CUR error}

        \addplot[
            color=signalorange,
            dashed,
            line width=1.5pt,
            mark=x,
            mark options={solid, scale=1.5},
        ]
        table[
                x=s_vec,
                y=tmiCURr,
                col sep=comma
            ]{csv-paper/quality-exp6/summary.csv};
        \addlegendentry{Estimate}

    \end{axis}
\end{tikzpicture}
}
        \caption{Exponential decay.}
        \label{fig:exp-decay}
    \end{subfigure}
    \hfill
    \begin{subfigure}{0.48\textwidth}
        \centering
        \scalebox{0.7}{
\begin{tikzpicture}
    \begin{axis}[
            width=\textwidth,
            height=8cm,
            xlabel={Rank},
            ymode=log,
            grid=both,
            legend to name={dummy},  
        ]

        \addplot[
            color=gray,
            solid,
            line width=1pt,
            mark=o,
            mark options={solid, scale = 0.7, fill = black},
        ]
        table[
                x=s_vec,
                y=errSVD,
                col sep=comma
            ]{csv-paper/qualityp15-frankenstein/summary.csv};
        \addlegendentry{Spectrum}

        \addplot[
            color=oxfordnavy,
            solid,
            line width=1.5pt,
            mark=o,
            mark options={fill=white, scale = 1.5}
        ]
        table[
                x=s_vec,
                y=errCUR,
                col sep=comma
            ]{csv-paper/qualityp15-frankenstein/summary.csv};
        \addlegendentry{CUR error}

        \addplot[
            color=signalorange,
            dashed,
            line width=1.5pt,
            mark=x,
            mark options={solid, scale=1.5},
        ]
        table[
                x=s_vec,
                y=tmiCURr,
                col sep=comma
            ]{csv-paper/qualityp15-frankenstein/summary.csv};
        \addlegendentry{Estimate}

    \end{axis}
\end{tikzpicture}
}
        \caption{Frankenstein.}
        \label{fig:frankestein}
    \end{subfigure}
    \caption{Estimation accuracy. True error of a CUR approximation and root-scale estimate of the error for exponentially decaying singular values (\textit{left}) and for the Frankenstein matrix (\textit{right}).}
    \label{fig:quality1}
\end{figure}
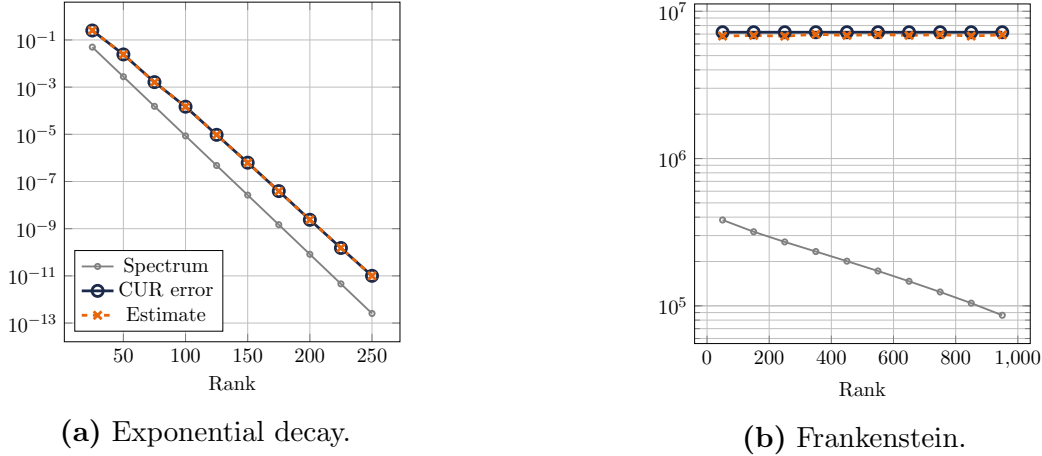
We begin by examining the behaviour of the estimator on a problem that is classically considered favourable for low-rank approximations,
namely a matrix $\mathbf{A}$ whose singular values decay exponentially.
The rapid spectral decay leads to highly accurate low-rank approximations,
and therefore represents a natural starting point for validating the proposed estimator.
We construct a $1000 \times 1000$ matrix $\mathbf{A}$ via its singular value decomposition,
imposing Haar distributed singular vectors and singular values $\sigma_i = 2^{-i/6}$.
We approximate $\mathbf{A}$ with a CUR approximation computed via column-pivoted QR for increasing target rank, and sample uniformly $q = 5$ extra columns.
Figure~\ref{fig:exp-decay} reports the average over $100$ independent trials of this experiment.
We observe that the estimate of the CUR error closely tracks the true error across all ranks considered.

\paragraph{Frankenstein matrix.}
To test the estimator in a more challenging setting, we construct a ``Frankenstein'' matrix,
that is, a $10000 \times 10000$ block diagonal matrix whose $2500 \times 2500$ blocks are
the adversarial test matrices \texttt{cauchy}, \texttt{golub}, \texttt{randcorr}, and \texttt{lotkin}
from Julia's \texttt{MatrixDepot}.
The CUR approximation is computed by selecting the first columns and rows of the matrix with increasing target rank,
which leads to a poor low-rank approximation.
We sample uniformly $q = 15$ extra columns.
Figure~\ref{fig:frankestein} reports the average over $100$ independent trials.
We observe that even in this highly adversarial setting,
and in contrast to the failure of the LRO estimator reported in Section~\ref{subsec:LRO-cur},
the proposed estimator accurately predicts the CUR error.
\paragraph{Spectro-microscopy experiment}

As already mentioned, the proposed estimator is of particular interest in settings where access to the matrix $\mathbf{A}$ is possible only through direct measurement of a small portion of its entries. A concrete example is the use of CUR approximations for subsampling in spectro-microscopy. In these experiments, we try to reconstruct the spatial and spectral identities of materials present in a sample. This is done via spectral scans of the sample at different levels of energy. Once the measurements are obtained, the analysis of such data is relatively fast and the true bottleneck of this experiment is the acquisition of the measurements themselves. However, it has been shown that the subsampling task arising in this (and possibly many other) synchrotron experiments can be reduced to a CUR approximation of a very low-rank matrix \cite{meier2026reducingacquisitiontimeradiation}. Since during the actual experiment only an approximate target rank is available, and restarting the entire experiment to refine the rank of the approximation is simply too expensive, an adaptive procedure would be highly desirable.
We therefore test our estimator on two spectro-microscopy datasets from \cite{meier_2026_18470992}. The first sample is of size $101\times 101$ and it is scanned for $149$ energy levels, while the second is of size $92\times 79$ and scanned for $150$ energy levels. We perform CURISS (as in \cite{meier2026reducingacquisitiontimeradiation}) on them, that is a problem specific way of finding good indices for a CUR of the unfolded tensor in which the rows correspond to a full energy scan and a column to a pixel over all energy levels.
Figure~\ref{fig:spectromicro} shows the average over $100$ trials of the CUR error and the proposed estimate measuring only two extra spatial rows (corresponding to columns in the appropriate  experiment representation) for these two datasets versus the subsampling ratio, i.e., how many measurements have been used to obtain a CUR approximation compared to full measurements. The results illustrate the practical relevance of the proposed estimator.
Indeed, the prediction of the error is very accurate, suggesting that an adaptive stopping criterion based on the proposed estimate would be both sensible and beneficial in practice.

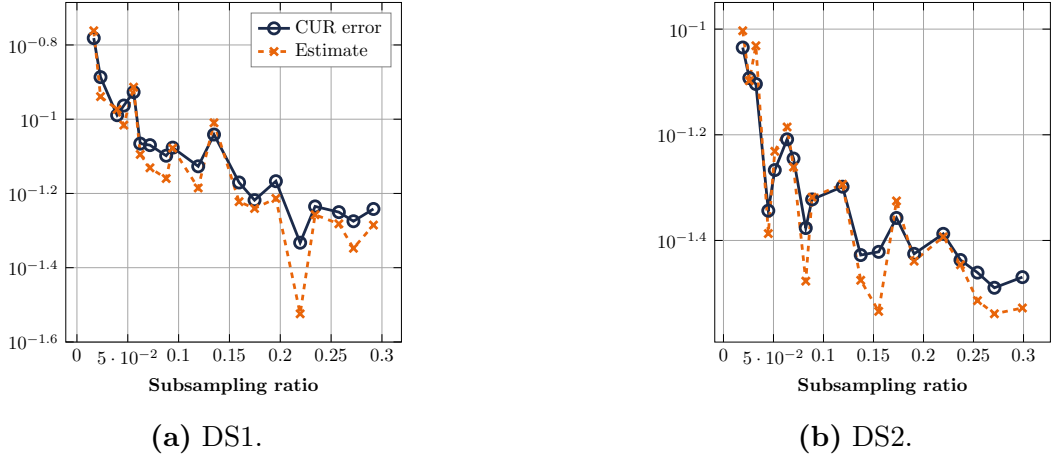
\begin{figure}
    \centering
    \begin{subfigure}[b]{0.48\textwidth}
        \centering
        \scalebox{0.7}{
\begin{tikzpicture}
    \begin{axis}[
            width=\textwidth,
            height=8cm,
            xlabel={Subsampling ratio},
            ymode=log,
            grid=both,
            grid style       = {line width=0.3pt, draw=black!20},
            major grid style = {line width=0.5pt, draw=black!35},
            axis background/.style = {fill=white},
            axis line style  = {black},
            tick style       = {black},
            tick label style = {font=\small, color=black},
            label style      = {font=\small\bfseries, color=black},
            legend pos       = north east,
            legend style     = {
                    font = \small,
                    fill = white,
                    draw = black!50,
                    text = black,
                },
            legend cell align = left,
        ]
        \addplot[
            color=oxfordnavy,
            solid,
            line width=1.5pt,
            mark=o,
            mark options={fill=white, draw=oxfordnavy, scale=1.5}
        ]
        table[
                x=p_actual,
                y=errCUR,
                col sep=comma
            ]{csv-paper/quality-DS1/summary.csv};
        \addlegendentry{CUR error}
        \addplot[
            color=signalorange,
            dashed,
            line width=1.5pt,
            mark=x,
            mark options={solid, scale=1.5},
        ]
        table[
                x=p_actual,
                y=tmiCURl_mean,
                col sep=comma
            ]{csv-paper/quality-DS1/summary.csv};
        \addlegendentry{Estimate}
    \end{axis}
\end{tikzpicture}
}
        \caption{DS1.}
        \label{fig:quality-ds1}
    \end{subfigure}
    \hfill
    \begin{subfigure}[b]{0.48\textwidth}
        \centering
        \scalebox{0.7}{
\begin{tikzpicture}
    \begin{axis}[
            width=\textwidth,
            height=8cm,
            xlabel={Subsampling ratio},
            ymode=log,
            grid=both,
            grid style       = {line width=0.3pt, draw=black!20},
            major grid style = {line width=0.5pt, draw=black!35},
            axis background/.style = {fill=white},
            axis line style  = {black},
            tick style       = {black},
            tick label style = {font=\small, color=black},
            label style      = {font=\small\bfseries, color=black},
            legend style={draw=none, fill=none},
            legend entries={},
        ]
        \addplot[
            color=oxfordnavy,
            solid,
            line width=1.5pt,
            mark=o,
            mark options={fill=white, draw=oxfordnavy, scale=1.5}
        ]
        table[
                x=p_actual,
                y=errCUR,
                col sep=comma
            ]{csv-paper/quality-DS2/summary.csv};
        \addplot[
            color=signalorange,
            dashed,
            line width=1.5pt,
            mark=x,
            mark options={solid, scale=1.5},
        ]
        table[
                x=p_actual,
                y=tmiCURl_mean,
                col sep=comma
            ]{csv-paper/quality-DS2/summary.csv};
    \end{axis}
\end{tikzpicture}
}
        \caption{DS2.}
        \label{fig:quality-ds2}
    \end{subfigure}
    \caption{Estimation accuracy. True error of a CUR approximation and root-scale estimate of the error for two spectro-microscopy datasets.}
    \label{fig:spectromicro}
\end{figure}
\subsection{Estimator uncertainty quantification}
In \cref{sec:uq}, we present two approaches to quantifying the uncertainty of the estimator. The first approach relies on asymptotic statistical theory to allow us to build an uncertainty interval using a Gaussian distribution. The second approach relies on the non-parametric bootstrap to form an uncertainty interval. We now investigate both of these approaches. 

To investigate the quality of these intervals we consider a CUR approximation of a matrix at different ranks. For each rank, we fix our CUR approximation and compute the true error. Additionally, we form 500 random sets of extra columns. For each set of extra columns, we apply each of our two uncertainty quantification techniques to form 95\% intervals. Finally, we repeat the formation of the 500 random sets of extra columns for value different value of extra columns, $q$,  and determine the percentage of the intervals that contain the true error. This is known as the coverage rate. We compute these coverage rates when a CUR approximation is applied to a $10,000 \times 10,000$ Chan matrix. We display the coverage rates for bootstrap sampling in \cref{tab:chan-cov-boot} and the coverage rates for the Gaussian interval in \cref{tab:chan-cov-gaus}. Additionally, to understand the widths of both intervals, we display the true error surrounded by a single instance of a both interval types formed with 10 extra columns in \cref{fig:bs-together}. We display these intervals both for the $10,000 \times 10,000$ Chan matrix and the Spectro-microscopy matrix.

\begin{figure}
    \centering
    \begin{subfigure}[b]{0.48\textwidth}
        \centering
%
\pgfplotsset{
    colormap={navygreen}{
            rgb255(0pt)=(27,42,74)       
            rgb255(33pt)=(15,110,86)     
            rgb255(66pt)=(61,170,110)    
            rgb255(100pt)=(232,101,10)   
            rgb255(133pt)=(253,246,227)  
        }
}
\definecolor{oxfordnavy}{RGB}{27,42,74}
\definecolor{darkteal}{RGB}{15,110,86}
\definecolor{leafgreen}{RGB}{61,170,110}
\definecolor{signalorange}{RGB}{232,101,10}
\definecolor{cyanteal}{RGB}{32,178,170}

\scalebox{0.7}{
\begin{tikzpicture}
    \begin{axis}[
            width        = \textwidth,
            height       = 8cm,
            xmin         = 0,
            xmax         = 1000,
            xtick        = {50,150,...,950},
            xticklabel style = {rotate=45, anchor=north east, font=\small},
            xlabel       = {target rank},
            ymode        = log,
            grid         = both,
            grid style       = {line width=0.3pt, draw=oxfordnavy!20},
            major grid style = {line width=0.5pt, draw=oxfordnavy!35},
            axis background/.style = {fill=white},
            axis line style  = {oxfordnavy},
            tick style       = {oxfordnavy},
            tick label style = {font=\small, color=oxfordnavy},
            label style      = {font=\small\bfseries, color=oxfordnavy},
            legend pos        = north east,
            legend style      = {
                    font = \small,
                    fill = white,
                    draw = oxfordnavy!50,
                    text = oxfordnavy,
                },
            legend cell align = left,
        ]

        \addplot [
            color       = oxfordnavy,
            line width  = 2pt,
            solid,
            mark        = triangle*,
            mark size   = 1pt,
            mark options= {fill=oxfordnavy, draw=none},
        ]
        table [x=s, y=estimate, col sep=comma]
            {csv-paper/bs-data-1/bci_gau.csv};
        \addlegendentry{Error CUR}


        \addplot [
            color       = cyanteal,
            line width  = 2.5pt,
            solid,
        ]
        table [x=s, y=lower, col sep=comma]
            {csv-paper/bs-data-1/bci_gau.csv};
        \addlegendentry{CI Bootstrap}

        \addplot [
            color       = red,
            line width  = 2.5pt,
            dashed,
        ]
        table [x=s, y=lower_n, col sep=comma]
            {csv-paper/bs-data-1/bci_gau.csv};
        \addlegendentry{CI Gaussian}
        
        \addplot [
            color       = cyanteal,
            line width  = 2.5pt,
            solid,
        ]
        table [x=s, y=upper, col sep=comma]
            {csv-paper/bs-data-1/bci_gau.csv};

        \addplot [
            color       =  red,
            line width  = 2.5pt,
            dashed,
        ]
        table [x=s, y=upper_n, col sep=comma]
            {csv-paper/bs-data-1/bci_gau.csv};
    \end{axis}
\end{tikzpicture}
}
        \caption{Chan matrix.}
        \label{fig:bs}
    \end{subfigure}
    \hfill
    \begin{subfigure}[b]{0.48\textwidth}
        \centering
%
\pgfplotsset{
    colormap={navygreen}{
            rgb255(0pt)=(27,42,74)       
            rgb255(33pt)=(15,110,86)     
            rgb255(66pt)=(61,170,110)    
            rgb255(100pt)=(232,101,10)   
            rgb255(133pt)=(253,246,227)  
        }
}
\definecolor{oxfordnavy}{RGB}{27,42,74}
\definecolor{darkteal}{RGB}{15,110,86}
\definecolor{leafgreen}{RGB}{61,170,110}
\definecolor{signalorange}{RGB}{232,101,10}
\definecolor{cyanteal}{RGB}{32,178,170}

\scalebox{0.7}{
\begin{tikzpicture}
    \begin{axis}[
            width        = \textwidth,
            height       = 8cm,
            xmin         = 0,
            xmax         = 19,
            xtick        = {1,2,...,19},
            xticklabel style = {rotate=45, anchor=north east, font=\small},
            xlabel       = {target rank},
            ymode        = log,
            grid         = both,
            grid style       = {line width=0.3pt, draw=oxfordnavy!20},
            major grid style = {line width=0.5pt, draw=oxfordnavy!35},
            axis background/.style = {fill=white},
            axis line style  = {oxfordnavy},
            tick style       = {oxfordnavy},
            tick label style = {font=\small, color=oxfordnavy},
            label style      = {font=\small\bfseries, color=oxfordnavy},
            legend pos        = north east,
            legend style      = {
                    font = \small,
                    fill = white,
                    draw = oxfordnavy!50,
                    text = oxfordnavy,
                },
            legend cell align = left,
        ]

        \addplot [
            color       = oxfordnavy,
            line width  = 2pt,
            solid,
            mark        = triangle*,
            mark size   = 1pt,
            mark options= {fill=oxfordnavy, draw=none},
        ]
        table [x=s, y=estimate, col sep=comma]
            {csv-paper/bs-data-1/bci_spectro.csv};
        \addlegendentry{Error CUR}


        \addplot [
            color       = cyanteal,
            line width  = 2.5pt,
            solid,
        ]
        table [x=s, y=lower, col sep=comma]
            {csv-paper/bs-data-1/bci_spectro.csv};
        \addlegendentry{CI Bootstrap}

        \addplot [
            color       = red,
            line width  = 2.5pt,
            dashed,
        ]
        table [x=s, y=lower_n, col sep=comma]
            {csv-paper/bs-data-1/bci_spectro.csv};
        \addlegendentry{CI Gaussian}
        
        \addplot [
            color       = cyanteal,
            line width  = 2.5pt,
            solid,
        ]
        table [x=s, y=upper, col sep=comma]
            {csv-paper/bs-data-1/bci_spectro.csv};

        \addplot [
            color       =  red,
            line width  = 2.5pt,
            dashed,
        ]
        table [x=s, y=upper_n, col sep=comma]
            {csv-paper/bs-data-1/bci_spectro.csv};
    \end{axis}
\end{tikzpicture}
}
        \caption{Spectro-microscopy matrix.}
        \label{fig:bs-width}
    \end{subfigure}
    \caption{Confidence intervals. CUR error, root-scale estimate and confidence intervals via bootstrap and Gaussian Estimation for the Chan matrix (\textit{Left}) and for a Spectro-microscopy experiment (\textit{Right}).}
    \label{fig:bs-together}
\end{figure}
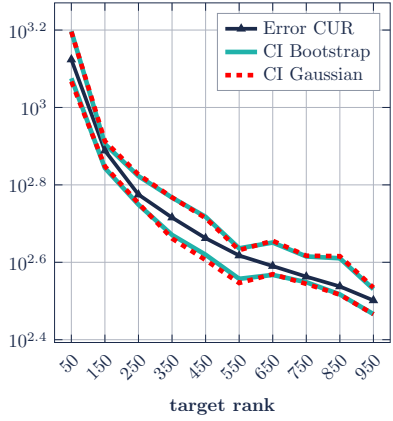
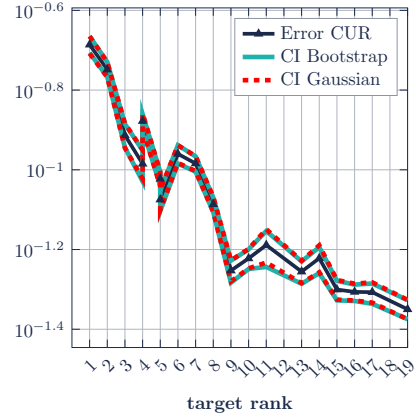
Looking at \cref{tab:chan-cov-boot} and \cref{tab:chan-cov-gaus}, we can see that generally both intervals under cover the designed coverage percentage of 95\%. The degree of this undercoverage is fairly similar for both the bootstrap and Gaussian intervals, as long as the number of extra columns is at least 10. At $q=5$, the Gaussian interval offers much better coverage than the bootstrap, although this coverage is still substantially below the nominal coverage percentage of 95\%. At higher numbers of extra columns, the intervals essentially have the same coverage. The undercoverage is not ideal, but it is also not to the extent where the intervals lack practical utility. 
Looking at \cref{fig:bs-together}, it is worth noting that neither interval is excessively wide compared to the actual error. Combining this observation with the coverage information seems to suggest; that one could probably use the upper bounds of these intervals for making stopping decisions. 
It should be noted that the Gaussian interval is substantially cheaper to compute than the bootstrap interval and performs similarly to the bootstrap for the Chan example.

\begin{table}[h]
    \centering
    \begin{tabular}{c|c|c|c|c|c|c|}
         Approx Rank & $q = 5$ & $q = 10$ & $q=15$ & $q = 20$ & $q = 25$ & $q = 30$ \\
         \midrule
         50 & 81.2& 90.6 & 91.0 & 94.2 & 93.4 & 94.4\\
         100 & 84.2& 93.0 & 92.6 & 94.2 & 94.0 & 95.0\\
         150 & 86.8& 90.4 & 91.6 & 95.0 & 93.0 & 92.8\\
         200 & 85.8 & 89.4 & 91.4 & 93.2 & 93.4 & 93.2\\
         250 & 83.6 & 87.6 & 91.4 & 91.6 & 93.0 & 93.4\\
         300 & 87.4 & 92.8 & 94.2 & 93.0 & 96.0 & 95.0\\
         350 & 85.6 & 91.2 & 95 & 94.0 & 94.2 & 94.2\\
         400 & 85.6 & 88.8 & 90.6 & 94.4 & 92.0 & 93.8\\
         450 & 78.8 & 87.6 & 93.4 & 93.0 & 91.8  & 92.6\\
         500 & 84.6 & 89.4 & 91.8 & 91.2 & 92.0 & 92.0\\
         \textbf{Avg.} & \textbf{84.36} & \textbf{90.1} & \textbf{92.3} & \textbf{93.4} & \textbf{93.3} & \textbf{93.6}
    \end{tabular}
    \caption{Table showing the coverage rates of the 95\% bootstrap confidence interval with different numbers of extra columns. Each percentage is out of 500 intervals.}
    \label{tab:chan-cov-boot}
\end{table}

\begin{table}[h]
    \centering
    \begin{tabular}{c|c|c|c|c|c|c|}
         Approx Rank & $q = 5$ & $q = 10$ & $q=15$ & $q = 20$ & $q = 25$ & $q = 30$ \\
         \midrule
         50 & 84.8& 91.4 & 92.0 & 94.8  & 93.0 & 94.0\\
         100 & 87.0 & 92.2 & 92.2 & 93.8 & 93.6 & 93.6\\
         150 & 89.4 & 92.2 & 92.6 & 95.6 & 93.4 & 94.2\\
         200 & 87.8 & 89.8 & 90.4 & 92.2 & 93.0 & 93.0\\
         250 & 88.2 & 91.0 & 92.0 & 93.4 & 94.0 & 93.0\\
         300 & 87.4 & 92.4 & 93.4 & 92.0 & 95.6 & 95.2\\
         350 & 87.6 & 90.4 & 94.2 & 93.4 & 93.2 & 94.6\\
         400 & 90.0 & 89.8 & 90.2 & 93.8  & 92.4 & 94.4\\
         450 & 83.0 & 89.2 & 93.8 & 93.4 & 93.2 & 92.8\\
         500 & 88.6 & 90.4 & 92.2 & 92.6 & 94.0 & 92.6\\
         \textbf{Avg.} & \textbf{87.4} & \textbf{90.9} & \textbf{92.3} & \textbf{93.5} & \textbf{93.5} & \textbf{93.7}
    \end{tabular}
    \caption{Table showing the coverage rates of the 95\% Gaussian interval with different numbers of extra columns. Each percentage is out of 500 intervals.}
    \label{tab:chan-cov-gaus}
\end{table}

\section{Discussion}
In this paper, we have considered the problem of estimating the error of a CUR approximation when the approximated matrix can only be accessed through a subset of its entries.
After discussing the limitations of existing leave-one-out–type estimators when applied to CUR approximations, we introduced an estimator that uses a small number of additional columns, beyond those already employed in the approximation, to sample the residual. We analyzed this estimator by relating its variance to the probability distribution used to select the extra columns and by highlighting the fundamental limitations imposed by the considered access constraints. To assess its practical reliability, we further proposed a bootstrap-based procedure.

Numerical experiments, including real-world applications such as spectro-microscopy, illustrate that a small number of uniformly sampled additional columns is generally sufficient to reliably achieve good estimation accuracy.

This work represents an important step toward the practical deployment of CUR approximations in scenarios where global access to the matrix is unavailable. Moreover, it opens the way to the development of adaptive CUR methodologies under restricted-access settings.

Future research directions include incorporating a rejection strategy into the selection of the extra columns in order to increase the probability of sampling residual components associated with subspaces not captured by the columns used in the approximation, and the derivation of refinement procedures that remain compatible with the restricted-access constraint.
\section*{Acknowledgments}
The authors would like to thank D. Halikias and Y. Nakatsukasa for their invaluable suggestions during the writing of this paper.
\bibliographystyle{plainnat}
\bibliography{references_nlaa}

\end{document}